\documentclass[11pt, reqno, english]{amsart}  
\usepackage[utf8]{inputenc}
\usepackage[T1]{fontenc}
\usepackage{amsmath,amsthm}
\usepackage{amsfonts,amssymb}
\usepackage{url}
\usepackage{mathtools}  
\usepackage[colorlinks=true,urlcolor=blue,linkcolor=red,citecolor=magenta]{hyperref}
\usepackage{enumerate, paralist}
\usepackage{graphicx}
\usepackage{tikz-cd}
\usepackage[left=1in,right=1in,top=1in,bottom=1in]{geometry}

\newcounter{commentcounter}

\theoremstyle{plain}
\newtheorem{theorem}{Theorem}[section]
\newtheorem{lemma}[theorem]{Lemma}
\newtheorem{corollary}[theorem]{Corollary}
\newtheorem{proposition}[theorem]{Proposition}
\newtheorem{conjecture}[theorem]{Conjecture}

\theoremstyle{definition}
\newtheorem{definition}[theorem]{Definition}
\newtheorem{example}[theorem]{Example}
\newtheorem{remark}[theorem]{Remark}
\newtheorem{problem}[theorem]{Problem}

\newtheorem*{assumption}{Assumption}

\newcommand{\R}{\mathbb{R}}
\newcommand{\C}{\mathbb{C}}
\newcommand{\Z}{\mathbb{Z}}

\newcommand{\KG}{\mathrm{KG}}
\newcommand{\Emb}{\mathrm{Emb}}
\newcommand{\AEmb}{\mathrm{AEmb}}

\DeclareMathOperator{\coind}{\mathrm{coind}}

\begin{document}

\title[Coupled embeddability]{Coupled embeddability}



\author{Florian Frick}
\address[FF]{Dept.\ Math.\ Sciences, Carnegie Mellon University, Pittsburgh, PA 15213, USA}
\email{frick@cmu.edu} 

\author{Michael Harrison}
\address[MH]{Dept.\ Math.\ Sciences, Carnegie Mellon University, Pittsburgh, PA 15213, USA}
\email{mah5044@gmail.com} 

\thanks{FF was supported by NSF grant DMS 1855591 and a Sloan Research Fellowship.  MH was supported by Mathematisches Forschungsinstitut Oberwolfach with an Oberwolfach Leibniz Fellowship.}


\begin{abstract}
\small
We introduce the notion of coupled embeddability,  defined for maps on products of topological spaces.  We use known results for nonsingular biskew and bilinear maps to generate simple examples and nonexamples of coupled embeddings.  We study genericity properties for coupled embeddings of smooth manifolds, extend the Whitney embedding theorems to statements about coupled embeddability, and we discuss a Haefliger-type result for coupled embeddings. We relate the notion of coupled embeddability to the $\Z/2$-coindex of embedding spaces, recently introduced and studied by the authors.  With a straightforward generalization of these results, we obtain strong obstructions to the existence of coupled embeddings in terms of the combinatorics of triangulations.  In particular, we generalize nonembeddability results for certain simplicial complexes to sharp coupled nonembeddability results for certain pairs of simplicial complexes.
\end{abstract}

\date{\today}
\maketitle

\section{Introduction}

Let $X$ and $Y$ be topological spaces such that $X$ does not embed in $\R^m$ and $Y$ does not embed in~$\R^n$.  Then for any continuous map $f \colon X \times Y \to \R^{m+n}$ and any decomposition $\R^{m+n} = \R^m \times \R^n$,
\begin{compactitem}
\item for fixed $y \in Y$, the projection of $f(-,y)$ to $\R^m$ fails to be an embedding, and
\item for fixed $x \in X$, the projection of $f(x,-)$ to $\R^n$ fails to be an embedding.
\end{compactitem}
When can the non-embeddability of $X$ into $\R^m$ and $Y$ into $\R^n$ be witnessed simultaneously?  We formalize this question as follows:

\begin{definition} Given $X$, $Y$, and $f$ as above, we say that $x_1, x_2 \in X$ and $y_1, y_2 \in Y$ form an \emph{axis-aligned parallelogram} if there exists a decomposition $\R^{m+n} = \R^m \times \R^n$ (not necessarily orthogonal) such that
\begin{compactitem}
\item $f(x_1,y_1) = f(x_2,y_1)$ and $f(x_1,y_2) = f(x_2,y_2)$ in the first $m$ coordinates, and
\item $f(x_1,y_1) = f(x_1,y_2)$ and $f(x_2,y_1) = f(x_2,y_2)$ in the last $n$ coordinates.
\end{compactitem}
In this case we say that $f$ is a \emph{coupled nonembedding}, or that it satisfies the parallelogram condition.  (See Figure \ref{fig:par}).
\end{definition}

\begin{figure}[ht!]
\centerline{
\includegraphics[width=4in]{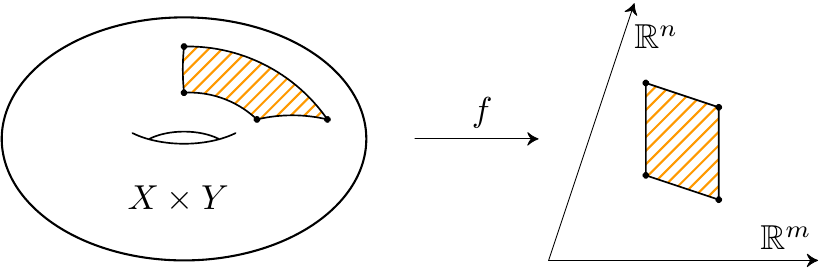}
}
\caption{An axis-aligned parallelogram}
\label{fig:par}
\end{figure}

\begin{problem} 
\label{prob:general}
Given $f \colon X \times Y \to \R^{m+n}$, do there exist $x_1, x_2 \in X$ and $y_1, y_2 \in Y$ which form an axis-aligned parallelogram?  
\end{problem}

We are particularly interested in a special case of coupled nonembeddability, which in fact initially motivated the general question above.

\begin{problem}
\label{prob:spheres}
Given $f \colon S^m \times S^n \to \R^{m+n}$, do there exist $x \in S^m$ and $y \in S^n$ such that $\pm x$, $\pm y$ form an axis-aligned parallelogram?
\end{problem}
This special case is natural from the perspective of equivariant topology.  Indeed, it follows from the Borsuk-Ulam theorem that for any decomposition $\R^{m+n} = \R^m \times \R^n$, 
\begin{compactitem}
\item for fixed $y \in S^n$, there exists $x' \in S^m$ such that $f(x',y) = f(-x',y)$ in the first $m$ coordinates, and
\item for fixed $x \in S^m$, there exists $y' \in S^n$ such that $f(x,y') = f(x,-y')$ in the last $n$ coordinates.
\end{compactitem}
Thus the problem above asks to determine whether the upshot of Borsuk-Ulam is simultaneously satisfied among $\pm x$ and $\pm y$.   We address this special case of the problem in Section \ref{sec:spheres}.

Before studying these problems we offer some simple remarks and clarifications.  Consider a codomain of general dimension $d$ and the space of maps $X \times Y \to \R^d$.  The coupled nonembeddability condition may be viewed as a certain degeneracy condition; for example if $d$ is large enough then a generic map $f$ does not satisfy a parallelogram condition with respect to \emph{any} decomposition $\R^d = \R^k \times \R^{d-k}$; in this case we say that $f$ is a \emph{coupled embedding}. 

We also comment on the role of the dimension $k$ in the decomposition $\R^d = \R^k \times \R^{d-k}$.  If there exists any axis-aligned parallelogram, for any $k \in \left\{ 0, \dots, d \right\}$, then there also exists an axis-aligned parallelogram with respect to some decomposition $\R^d = \R^\ell \times \R^{d - \ell}$ for each $\ell \in \left\{ 1, \dots, d-1 \right\}$, by allowing each factor of the decomposition to contain the span of one edge of the parallelogram.  Thus (assuming one ignores the cases $k = 0$ and $k=d$) it is not necessary to specify the dimensions of the decomposition; that is, the coupled embedding condition depends only on $X$, $Y$, and $d$.

\begin{problem} \label{prob:mindim} Given $X$ and $Y$,  what is the minimum dimension $d(X,Y)$ such that there exists a coupled embedding $f \colon X \times Y \to \R^{d(X,Y)}$?
\end{problem}

This question has been posed and studied for other nondegeneracy conditions: immersions, embeddings, $k$-regular embeddings, totally skew embeddings, totally nonparallel immersions, and more; yet these are all conditions on a single manifold.   To the best of our knowledge there has been no study of differential conditions on product manifolds which take into account the product structure.  We find Problem \ref{prob:mindim} enticing both due to its simplicity as a condition on product manifolds, and due to the undertones of equivariant topology.

\begin{remark}  We will see that if $e_X$ and $e_Y$ represent, respectively, the minimum dimensions of embeddability of $X$ and $Y$, then $d(X,Y) < e_X + e_Y$.  Thus the original nonembeddability assumptions on $X$ and $Y$, while not strictly necessary, are mentioned to focus the attention on the interesting range of dimensions for this problem.
\end{remark}

\begin{remark}  It is possible that $X$ does not embed in $\R^m$ and $Y$ does not embed in $\R^n$, yet nevertheless there exists an embedding $X \times Y \to \R^{m+n}$.  For example, $\R P^2$ does not embed in $\R^3$, and $\R P^3$ does not embed in $\R^4$, yet their product embeds in $\R^7$ (see \cite[Lemma 2.1]{ARS}).
\end{remark}

Our main results for coupled embeddings are as follows:

\begin{theorem}
\label{thm:mainexist}
Let $M$ and $N$ be smooth compact manifolds of dimensions $p$ and $q$.  Then:
\begin{compactenum}[(a)]
\item For $d > 2(p+q)$, a generic map $f \in C^\infty(M \times N, \R^d)$ is a coupled embedding.
\item There exists a coupled embedding $M \times N \to \R^{2p+2q-2}$.
\end{compactenum}
\end{theorem}

It is interesting to compare this result to the Whitney theorems, which state that every smooth $p$-dimensional manifold $M$ embeds in $\R^{2p}$, and that if $d > 2p$, a generic map $f \in C^\infty(M, \R^d)$ is an embedding.  Thus it is somewhat surprising to see a difference of \emph{three} dimensions between the genericity result and the existence result for coupled embeddings.

\begin{problem} Given $p$ and $q$, what is the minimum dimension $d$ such that for every $p$-dimensional manifold $M$ and $q$-dimensional manifold $N$, there exists a coupled embedding $M \times N \to \R^d$?
\end{problem}

The answer to this question is unknown even for embeddings of manifolds, but if $p$ is not a power of $2$, then every $p$-dimensional manifold embeds in a dimension less than $2p$.  Similarly, we will show that the minimum dimension $d$ for coupled embeddings of manifolds satisfies $d < 2p+2q-2$ unless one of $p$ or $q$ is a power of $2$.  

The question is easier for embeddings of simplicial complexes of dimension $p$, for which the answer is $2p+1$.  This is sharp due to the nonembeddability of $[3]^{*(p+1)}$ and $\Delta^{(p)}_{2p+2}$ into $\R^{2p}$.  We will show that under a certain number-theoretic condition on $p$ and $q$, a similar result holds for coupled embeddings.

\begin{theorem}
\label{thm:mainobstruct}
Suppose that the nonnegative integers $p$ and $q$ do not share a one in any digit of their binary expansions.  Then $2p+2q+1$ is the minimum dimension into which every pair of simplicial complexes coupled embeds.  In particular,
\[
d\left([3]^{*(p+1)}, [3]^{*(q+1)}\right) = d\left(\Delta^{(p)}_{2p+2}, [3]^{*(q+1)}\right) = d\left(\Delta^{(p)}_{2p+2}, \Delta^{(q)}_{2q+2}\right) = 2p+2q + 1.
\]
\end{theorem}

The structure of the paper is as follows.  In Section \ref{sec:prelim} we review definitions for biskew and bilinear maps, which play a crucial role in the study of coupled embeddings.   In Section \ref{sec:existence} we extend usual transversality results to discuss genericity statements for coupled embeddings and to prove Theorem \ref{thm:mainexist}.  We also discuss a Haefliger-type conjecture for coupled embeddings and outline a proposed argument.  In Section \ref{sec:nonexistence} we study coupled embeddability of simplicial complexes.  We discuss a process for obtaining lower bounds on $d(X,Y)$ in terms of combinatorics of triangulations of $X$ and $Y$.   We prove Theorem \ref{thm:mainobstruct} and compute the value of $d$ for other pairs of spaces.

\section{Preliminaries and basic bounds}
\label{sec:prelim}

In this section we recall the definitions and basic results for biskew and bilinear maps.  In Section~\ref{sec:spheres} we will see that Problem \ref{prob:spheres} can be easily restated in terms of the existence of nonsingular biskew maps, allowing us to focus exclusively on the more general Problems \ref{prob:general} and \ref{prob:mindim}.

\subsection{Nonsingular biskew and bilinear maps}

We recall the definitions and basic results for \emph{biskew}, i.e.\ $(\Z/2)^2$-equivariant, maps $S^m \times S^n \to \R^d$; where here each generator acts by negation on the codomain.  A biskew map is called \emph{nonsingular} if $0$ is not in its image.   Given $m$ and $n$, the smallest dimension $d$ such that there exists a nonsingular biskew map $S^m \times S^n \to \R^d$ is unknown; however, the following obstruction is well-known.

\begin{lemma}
\label{lem:binary} 
If $m$ and $n$ do not share a one in any digit of their binary expansions, there is no nonsingular biskew map $S^m \times S^n \to \R^{m+n}$.  That is, every such map hits $0$.
\end{lemma}

A simple method of proof, originally due to Hopf \cite{Hopf} (though not originally presented in the modern language of algebraic topology), is to assume the existence of a nonsingular biskew map, write down the induced map in cohomology $H^*(\R P^{m+n-1}) \to H^*(\R P ^{m-1}) \otimes H^*(\R P^{n-1})$, and deduce the resulting number theoretic condition on $m$ and $n$ by studying the image of the generator.

These considerations are important for problems of immersions of projective spaces (see e.g.\ \cite[Section 6]{James}).

A bilinear map $B \colon \R^{m+1} \times \R^{n+1} \to \R^d$ is called \emph{nonsingular} if $B(x,y) = 0$ implies $x=0$ or $y=0$.  As with biskew maps, the smallest dimension $d$ for which there exists a nonsingular bilinear map $\R^{m+1} \times \R^{n+1} \to \R^d$ is, in general, unknown.  Nonsingular bilinear maps appeared a century ago in the works of Hurwitz and Radon in their studies of square identities (\cite{Hurwitz, Radon}), and since then, they have made prominent appearances in topology.  For example, a famous result of Adams~\cite{Adams} states that there exists a nonsingular bilinear map $\R^p \times \R^q \to \R^p$ if and only if there exist $q-1$ linearly independent tangent vector fields on $S^{p-1}$.  More recently, Ovsienko and Tabachnikov~\cite{OvsienkoTabachnikov} showed that these statements are equivalent to the existence of a fibration of $\R^{p+q-1}$ by pairwise skew affine copies of $\R^{q-1}$ (see also \cite{Harrison, Harrison2, Harrison3,OvsienkoTabachnikov2}).  The existence of nonsingular bilinear maps was studied in a series of articles by K.Y.\ Lam (e.g.\ \cite{Lam1, Lam6, Lam4, Lam3, Lam2, Lam5}) and by Berger and Friedland~\cite{BergerFriedland}.

We only state some rudimentary existence results.

\begin{lemma}
\label{lem:nonsingular}
There exists a nonsingular bilinear map $\R^{m+1} \times \R^{n+1} \to \R^{m+n+1}$.  If $m$ and $n$ are both odd,  there exists a nonsingular bilinear map $\R^{m+1} \times \R^{n+1} \to \R^{m+n}$.
\end{lemma}

\begin{proof}  Regard each $\R^k$ as the coefficients of degree-$(k-1)$ polynomials.  Polynomial multiplication then yields a nonsingular bilinear map $\R^{m+1} \times \R^{n+1} \to \R^{m+n+1}$.  When $m$ and $n$ are odd, complex polynomial multiplication yields the desired nonsingular bilinear map.
\end{proof}

Other common examples of nonsingular bilinear maps include those induced by complex multiplication (that is, $\R^2 \times \R^{2k} \to \R^{2k}$) and similarly for quaternionic and octonionic multiplication.

Since a nonsingular bilinear map induces a nonsingular biskew map, Lemma \ref{lem:nonsingular} may be used for existence results for bilinear and biskew maps, and Lemma \ref{lem:binary} may be used for nonexistence results for bilinear and biskew maps.

\subsection{Coupled embeddings induced by embeddings}

We first observe the following.

\begin{lemma}
\label{lem:nonpar}
A bilinear map $B \colon \R^{m+1} \times \R^{n+1} \to \R^d$ is a coupled embedding if and only if $B$ is nonsingular.
\end{lemma}

\begin{proof} A map $B$ admits an axis-aligned parallelogram $x_1, x_2, y_1, y_2$ if and only if $B(x_1,y_1) + B(x_2,y_2) = B(x_2,y_1) + B(x_1,y_2)$, which is equivalent to the condition $B(x_1 - x_2, y_1 - y_2) = 0$ by bilinearity of $B$.
\end{proof}

In this sense coupled embeddability could be viewed as a natural generalization of nonsingularity for bilinear maps.

\begin{corollary}
\label{cor:dim}
Suppose that $e_X$ (resp.\ $e_Y$) is the minimum dimension of embeddability of $X$ (resp.~$Y$).  Then $d(X,Y)$ is at most the minimum dimension into which there exists a  nonsingular bilinear map from $\R^{e_X} \times \R^{e_Y}$.  In particular, $d(X,Y) \leq e_X + e_Y - 1$.
\end{corollary}

We will use this result frequently to obtain upper bounds on the quantity $d(X,Y)$.  However, we also note that this result justifies the original focus on dimensions $d$ of the form $m + n$, such that $X$ does not embed in $\R^m$ and $Y$ does not embed in $\R^n$.

\subsection{Biskew maps and Problem \ref{prob:spheres}}
\label{sec:spheres}

Here we address the fact that we have stated two problems for maps $S^m \times S^n$, one which asks for general axis-aligned parallelograms, and one which asks for axis-aligned parallelograms of the specific form $\pm x$, $\pm y$.  This subsection addresses the latter problem, which may be completely answered in terms of nonsingular biskew maps $S^m \times S^n \to \R^d$.  To avoid confusion, we say that a map $S^m \times S^n \to \R^d$ is a \emph{coupled} $\Z/2$-\emph{embedding} if it has no axis-aligned parallelograms of the specific form $\pm x$, $\pm y$.

\begin{proposition}
There exists a coupled $\Z/2$-embedding $S^m \times S^n \to \R^d$ if and only if there exists a nonsingular biskew map $S^m \times S^n \to \R^d$.
\end{proposition}

\begin{proof}
Given $f \colon S^m \times S^n \to \R^d$,  for $m, n \geq 1$, define
\[
\Phi_f \colon S^m \times S^n \to \R^d \colon (x,y) \mapsto f(x,y) + f(-x,-y) - f(x,-y) - f(-x,y),
\]
and observe that $f$ is a coupled $\Z/2$-embedding if and only if the $(\Z/2)^2$-equivariant map $\Phi_f$ avoids~$0$.  This establishes the backward implication, and if $f$ is a nonsingular biskew map, then $\Phi_f = 4f$ is nonsingular.
\end{proof}

\begin{corollary}
If $m$ and $n$ do not share a one in any digit of their binary expansions, then there is no coupled $\Z/2$-embedding $S^m \times S^n \to \R^{m+n}$.
\end{corollary}

The next corollary is a consequence of Lemma \ref{lem:nonsingular}.

\begin{corollary}
If $m$ and $n$ are both odd, then there exists a coupled $\Z/2$-embedding $S^m \times S^n \to \R^{m+n}$.
\end{corollary}

For example, there is no coupled $\Z/2$-embedding $S^1 \times S^2 \to \R^3$,  but there are coupled $\Z/2$-embeddings $S^1 \times S^2 \to \R^4$ and $S^1 \times S^1 \to \R^2$.   Further specific examples can be written using Corollary \ref{cor:dim} and the existence of nonsingular bilinear maps (see e.g.\ the table in \cite{BergerFriedland}).

\begin{remark}
It is unclear whether the two coupled embedding formulations for $S^m \times S^n$ are equivalent.  For example, Gitler and Lam showed in \cite{GitlerLam} that there exists a nonsingular biskew map $S^{27} \times S^{12} \to \R^{32}$, but that there is no nonsingular bilinear map $\R^{28} \times \R^{13} \to \R^{32}$.  Thus there is a coupled $\Z/2$ embedding $S^{27} \times S^{12} \to \R^{32}$ but we do not know whether there exists a map $S^{27} \times S^{12} \to \R^{32}$ with no general axis-aligned parallelogram.
\end{remark}

\subsection{Basic lower bounds}

Having disposed of Problem \ref{prob:spheres}, we now return to the more general Problem~\ref{prob:mindim}: given topological spaces $X$ and $Y$, we seek the minimum dimension $d(X,Y)$ such that there exists a coupled embedding $X \times Y \to \R^d$.  We have seen that $d(X,Y) \leq e_X + e_Y - 1$, where $e_X$ and $e_Y$ are the minimum embedding dimensions of $X$ and $Y$.  Here we present some results which give basic lower bounds for $d(X,Y)$.

Let $F_2(X)$ represent the (ordered) configuration space consisting of ordered pairs of unequal points of $X$, $F_2(X) = X \times X - \Delta X = \left\{ (x_1,x_2) \in X \times X \ \big| \ x_1 \neq x_2 \right\}$.

\begin{lemma} If there exists a coupled embedding $f \colon X \times Y \to \R^d$, then there exists a $(\Z/2)^2$-equivariant map $F_2(X) \times F_2(Y) \to \R^d$ which avoids zero.
\end{lemma}

\begin{proof}
Define
\[
\Phi_f \colon F_2(X) \times F_2(Y) \to \R^d \colon (x_1,x_2,y_1,y_2) \mapsto f(x_1,y_1) + f(x_2,y_2) - f(x_1,y_2) - f(x_2,y_1).
\]
The map $\Phi_f$ is $(\Z/2)^2$-equivariant, and if $f$ is a coupled embedding, there is no axis-aligned parallelogram, hence $\Phi_f$ avoids $0$.
\end{proof}

\begin{remark} We observe that if $f$ is a coupled embedding, then the map
\[
X \times F_2(Y) \to \R^d \colon (x,y_1,y_2) \mapsto f(x,y_1) - f(x,y_2)
\]
is an embedding of $X$ for fixed $y_1,y_2$, and the map
\[
F_2(X) \times Y \to \R^d \colon (x_1,x_2,y) \mapsto f(x_1,y) - f(x_2,y)
\]
is an embedding of $Y$ for fixed $x_1,x_2$. 
\end{remark}

\begin{corollary}
\label{cor:topdim}
Let $X$ and $Y$ be topological manifolds.  If there exists a coupled embedding $X~\times~Y~\to~\R^d$, then there exists a nonsingular biskew map $S^{\dim(X)-1} \times S^{\dim(Y)-1} \to \R^d$.
\end{corollary}

\begin{proof}  There is an embedding $\varphi \colon S^{\dim(X) - 1} \to X$,  hence a $\Z/2$-equivariant map $S^{\dim(X) - 1}~\to~F_2(X)$ sending $x \mapsto (\varphi(x),\varphi(-x))$, and similarly for $Y$.  A coupled embedding $f \colon X \times Y \to \R^d$ induces a nonsingular $(\Z/2)^2$-equivariant map $\Phi_f$, which by restriction, yields a nonsingular biskew map $S^{\dim(X)-1} \times S^{\dim(Y)-1} \to \R^d$.
\end{proof}

It is possible that a sphere of larger dimension than $\dim(X) - 1$ maps $\Z/2$-equivariantly into~$F_2(X)$.  It is useful to introduce the following terminology.

\begin{definition}  Let $Z$ be a space with a free $\Z/2$-action.  The \emph{$\Z/2$-coindex} of $Z$, denoted by~$\coind(Z)$, is the dimension of the largest-dimensional sphere which maps $\Z/2$-equivariantly into~$Z$.
\end{definition}

In this language, the Borsuk-Ulam theorem is (the nontrivial part of) the statement $\coind(S^k) =~k$.  The proof of Corollary \ref{cor:topdim} shows more generally:

\begin{proposition} If there exists a coupled embedding $X \times Y \to \R^d$, then there exists a nonsingular biskew map $S^{\coind(F_2(X))} \times S^{\coind(F_2(Y))} \to \R^d$.
\end{proposition}

\section{Coupled embeddings: existence}
\label{sec:existence}

In this section we address the existence question for coupled embeddings of smooth manifolds.  Let $M$ be a smooth manifold of dimension $p$.  Recall the weak Whitney theorems, which state that a generic map $M \to \R^{2p+1}$ is an embedding, and that a generic map $M \to \R^{2p}$ is an immersion; and the strong Whitney theorems,  which state that every smooth $p$-manifold $M$ embeds in $\R^{2p}$ and immerses into $\R^{2p-1}$.   We can make similar statements for coupled embeddings; however, the application is somewhat tricky because the coupled embedding condition is neither a condition on jet spaces nor on multijet spaces, where most functional conditions live.

For a reference on jet bundles, we recommend the book by Golubitsky and Guillemin \cite{GolubitskyGuillemin}.  We offer only some brief intuition.  Jet bundles are defined to mimic Taylor series on manifolds in a coordinate-free way.   The $k$-jet bundle $J^k(M, \R^d) \to M$ contains information about possible derivatives, up to order $k$, of functions $f \in C^\infty(M,\R^d)$.  The $k$-jet extension of $f$ is the section $j^kf \colon M \to J^k(M,\R^d)$, such that $j^kf(x)$ describes the Taylor expansion of $f$ at $x$, up to order $k$, in an invariant way.  Jet bundles allow for easy formalization of genericity statements such as the weak Whitney theorem on immersions.

Now consider the projection $\pi \colon J^k(M, \R^d) \times \cdots \times J^k(M,\R^d) \to M \times \cdots \times M$ ($s$ factors each), and let $J_s^k(M,\R^d) = \pi^{-1}(F_s(M))$.  The bundle $\pi \colon J_s^k(M,\R^d) \to F_s(M)$ is the \emph{$s$-fold $k$-multijet bundle} and allows for the study of differential conditions at $s$-tuples of points.   The $s$-fold $k$-multijet extension of $f$ is the section $j_s^k f(x_1,\dots,x_s) = (j^kf(x_1),\dots,j^kf(x_s))$, which simultaneously contains the information of derivatives up to order $k$ at $s$ distinct points of $M$.  Multijet bundles allow for easy formalization of genericity statements for conditions at pairs or $k$-tuples of points, such as the weak Whitney theorem on embeddings.

\subsection{Genericity statements for coupled embeddings}

Let $M$ and $N$ be smooth manifolds and let $Z = M \times N$.  The $2$-fold $0$-multijet space $J^0_2(Z, \R^d)$ is just $F_2(Z) \times \R^d \times \R^d$, and for $f \in C^\infty(Z, \R^d)$,  and the $2$-fold $0$-multijet extension of $f$ is $j_2^0 f(x_1,y_1,x_2,y_2) = (x_1,y_1,x_2,y_2,f(x_1,y_1),f(x_2,y_2))$.  This space is useful for studying embeddings $Z \to \R^d$.

We define a similar space which is useful for studying the coupled embedding condition.   Let $\Delta M$ be the diagonal in $M \times M$. Let $\Delta = (\Delta M \times N^2) \cup (M^2 \times \Delta N) \subset Z \times Z$, and let $C$ be its complement.
Consider the projection $\alpha \colon Z \times Z \times \R^{4d} \to Z \times Z$, and let $J = \alpha^{-1}(C) = C \times \R^{4d}$.  Given $f \in C^\infty(Z, \R^d)$, define the section
\[
\psi_f \colon C \to J \colon (x_1,y_1,x_2,y_2) \mapsto (x_1,y_1,x_2,y_2,f(x_1,y_1),f(x_2,y_2),f(x_1,y_2),f(x_2,y_1)).
\]

Recall that a subset $X \subset Y$ is \emph{residual} if it is a countable intersection of open dense subsets of~$Y$.

\begin{theorem}
\label{thm:notmultijet} Let $W$ be a submanifold of $J$.  Let
\[
T_W = \left\{ f \in C^\infty(Z, \R^d) \ \big| \ \psi_f \mbox{ is transverse to } W \right\}.
\]
Then $T_W$ is a residual subset of $C^\infty(Z,\R^d)$ in the Whitney $C^\infty$ topology.  Moreover, if $W$ is compact, then $T_W$ is open.
\end{theorem}

The statement is very similar to that of the Multijet Transversality Theorem, and a proof can be obtained by making appropriate modifications to the proof in \cite[Theorem II.4.13]{GolubitskyGuillemin}.  We give the essential ideas of this argument.

\begin{proof}[Proof sketch]

First assume that $W$ is compact.  Cover the compact set $\alpha(W) \cap C$ by finitely many rectangles of the form $\left\{C_\beta\right\}$, where $C_\beta = U^\beta_1 \times V^\beta_1 \times U^\beta_2 \times V^\beta_2$ for open $U^\beta_i \subset M$ and open $V^\beta_i \subset N$ satisfying the property that for all $\beta$, $\overline{U^\beta_1} \cap \overline{U^\beta_2} = \emptyset$ and $\overline{V^\beta_1} \cap \overline{V^\beta_2} = \emptyset$.   Define
\[
T_\beta = \left\{ f \in C^\infty(Z, \R^d) \ \big| \ \psi_f  \mbox{ is transverse to } W \mbox{ on } \alpha^{-1}(C_\beta) \right\}.
\]
It is enough to show that each $T_\beta$ is open.  By usual transversality results the set
\[
\left\{ g \in C^\infty(C, J) \ \big| \ g \mbox{ is transverse to } W \mbox{ on } \alpha^{-1}(C_\beta) \right\}
\]
is open, and taking the preimage under the continuous map $\psi \colon C^\infty(Z, \R^d) \to C^\infty(C, J)$ gives openness of $T_\beta$.

For general $W$, we may choose an open cover $\left\{W_i\right\}$ such that $\overline{W_i} \subset W$ and that $\alpha(\overline{W_i})$ may be covered by a rectangle $U^i_1 \times V^i_1 \times U^i_2 \times V^i_2$ as above.  It is enough to show that
\[
T_{W_i} = \left\{ f \in C^\infty(Z, \R^d) \ \big| \ \psi_f  \mbox{ is transverse to } W \mbox{ on } \overline{W_i} \right\}
\]
is open and dense, since then $T_W$ is a countable intersection of open dense sets, as desired.

By the above result, each $T_{W_i}$ is open.  To show density, we observe that we may perturb a function $f \colon Z \to \R^d$, using bump functions on the four disjoint sets $U^i_1 \times V^i_1$, $U^i_1 \times V^i_2$, $U^i_2 \times V^i_1$, and $U^i_2 \times V^i_2$, to achieve the desired transversality on $\overline{W_i}$.
\end{proof}

\begin{proof}[Proof of Theorem \ref{thm:mainexist}(a)]
Define
\[
\Sigma = \left\{ (x_1,y_1,x_2,y_2,z_1,z_2,z_3,z_4) \in C \times \R^{4d} \ \big| \ z_1 + z_2 - z_3 - z_4 = 0 \right\}.
\]
Observe that $\Sigma$ has codimension $d$ and that $f$ is a coupled embedding if and only if $\psi_f$ is disjoint from $\Sigma$.  By Theorem \ref{thm:notmultijet}, $\psi_f$ is transverse to $\Sigma$ for the members $f$ of some residual set, which is dense by Baire category theorem.   For such $f$, $(\psi_f)^{-1}(\Sigma)$ is a submanifold of $C$ of codimension $d$.  If $d > \dim(C) = 2(p+q)$ this set is empty, hence $\psi_f$ is disjoint from $\Sigma$.
\end{proof}

\begin{remark} Observe that for $d \leq 2(p+q)$, the same argument can be used to compute the expected dimension of the set on which coupled embeddability fails.
\end{remark}

We compare this to a result we could have obtained more easily from Whitney's theorem.  A generic map $M \to \R^{2p+1}$ is an embedding, and a generic map $N \to \R^{2q+1}$ is an embedding.  There exists a nonsingular bilinear map $B : \R^{2p+1} \times \R^{2q+1} \to \R^{2p+2q+1}$.  Therefore, for generic maps $f \in C^{\infty}(M,\R^{2p+1})$,  $g \in C^\infty(N,\R^{2q+1})$, the map $B \circ (f \times g) \colon M \times N \to \R^{2p+1} \times \R^{2q+1} \to \R^{2(p+q)+1}$ is a coupled embedding.   Of course, maps of the form $f \times g$ are not generic in $C^\infty(M \times N, \R^{2(p+q)+1})$, so this line of thought could not have given Proposition \ref{thm:mainexist}(a).

\begin{proof}[Proof of Theorem \ref{thm:mainexist}(b)]
There exists a nonsingular bilinear map $B \colon \R^{2p} \times \R^{2q} \to \R^{2p+2q-2}$ from complex polynomial multiplication.  Apply $B$ to embeddings $f \colon M \to \R^{2p}$, $g \colon N \to \R^{2q}$, which exist due to the strong Whitney embedding theorem.
\end{proof}

Surprisingly, this existence result guarantees a coupled embedding in a space of dimension \emph{three} less  than that in which coupled embeddings are generic.   Intuitively, one might expect the transition from generic to non-generic to yield only one less dimension, as is the case for the Whitney embedding theorems.  Here, the stronger result can be attributed both to the ``coupled'' nature of the coupled embedding condition, and also due to the better existence result for nonsingular bilinear maps when both dimensions are even.

When $n$ is a power of $2$,  the strong Whitney embedding theorem is sharp, in the sense that $\R P^n$ does not embed in $\R^{2n-1}$.  We ask:

\begin{problem}
Do there exist $p>1$, $q>1$, and smooth manifolds $M$ and $N$ of dimensions $p$ and~$q$,  such that $d(M,N) = 2p+2q-2$?
\end{problem}

If $M$ embeds in $\R^{2p-1}$ and $N$ embeds in $\R^{2q-1}$, which in particular must occur when neither $p$ and $q$ are powers of $2$, then we apply a nonsingular bilinear map to achieve coupled embeddability in $\R^{2p+2q-3}$.  If $p>1$ and $q>1$ are both powers of $2$, then $M$ embeds in $\R^{2p}$, $N$ embeds in $\R^{2q}$, and we apply a nonsingular bilinear map $\R^{2p} \times \R^{2q} \to \R^{2p+2q-4}$, which exists due to quaternionic polynomial multiplication  (see Example \ref{ex:rp2}).  Thus, the remaining possibility is that (without loss of generality) $M$ embeds in dimension $2p$ and not less (so $p$ must be a power of $2$) and $N$ embeds in dimension $2q-1$ and not less.  For example, this occurs with $M = \R P^2$ and $N = \R P^3$.

\begin{example}
\label{ex:rp2}
There exists a coupled embedding $\R P^2 \times \R P^2 \to \R^4$, since there is a nonsingular bilinear map $\R^4 \times \R^4 \to \R^4$ induced by quaternionic multiplication.
\end{example}

\subsection{A Haefliger-type conjecture for coupled embeddings}

The following result is due to Haefliger.

\begin{theorem}[Haefliger \cite{Haefliger}]  Let $M$ be a smooth closed manifold of dimension $n$.  If $d > \frac32 n + \frac32$, then the existence of a differentiable embedding $M \to \R^d$ is equivalent to the existence of a $\Z/2$-equivariant map $g \colon M \times M \to \R^d$ such that $g^{-1}(0) = \Delta M$.
\end{theorem}

We expect the following to hold for coupled embeddings.

\begin{conjecture} \label{con:haef} Let $M$ and $N$ be smooth closed manifolds of dimensions $p$ and $q$, respectively.  If $d~>~\frac32 ~(p~+~q)~+~\frac32$, then the existence of a differentiable coupled embedding $M \times N \to \R^d$ is equivalent to the existence of a $(\Z/2)^2$-equivariant map $g \colon M \times N \times M \times N \to \R^d$ such that $g^{-1}(0) = \Delta$, where $\Delta = (\Delta M \times N^2) \cup (M^2 \times \Delta N)$.
\end{conjecture}

Of course, the forward direction comes from the map $\Phi_f$ for a coupled embedding $f$.

We outline a possible method of proof, based on the ``removal of singularities'' h-principle technique due to Gromov and Eliashberg \cite{GromovEliashberg}.  See \cite{Harrison5} for an exposition of this technique, and to see its usefulness in proving a Whitney-type theorem for \emph{totally nonparallel immersions}.

Suppose that there exists a $(\Z/2)^2$-equivariant map $g \colon M \times N \times M \times N \to \R^d$ such that $g^{-1}(0)~=~\Delta$.  Observe that if $g = \Phi_f$ for some differentiable map $f \colon M \times N \to \R^d$, then $f$ is a coupled embedding, and the proof is complete.  The strategy is to replace the component functions of $g = (g_1,\dots,g_d)$, one by one, by functions of the form $\Phi_{f_i}$,  for $f_i \colon M \times N \to \R$, while preserving, at each stage, the condition that the preimage of $0$ is precisely $\Delta$.

Suppose that naively we replace $g_1$ by $\Phi_{f_1}$ for some arbitrary function $f_1 \colon M \times N \to \R$.  This causes no issues on the complement of the set
\[
\Sigma = (g_2,\dots,g_d)^{-1}(0),
\]
since regardless of the value of $\Phi_{f_1}$, the function $(\Phi_{f_1}, g_2, \dots, g_d)$ cannot map points of $\Sigma^c$ to $0$.  Therefore we only must be careful about the replacement at points of $\Sigma$.  For this, we recall that the function $g_1$ is safe on $\Sigma$, and we make use of this fact during the replacement.

We first focus on making the replacement on the complement $C$ of some open $(\Z/2)^2$-equivariant neighborhood of $\Delta$. Let $\pi \colon M \times N \times M \times N \to M \times N$ be the projection to the first factor.  We make the following assumption:

\begin{assumption}
The restriction of $\pi$ to $\Sigma \cap C$ is injective.
\end{assumption}

We claim that with this assumption, we can make the desired replacement of $g_1$ by $\Phi_{f_1}$.  First, observe that by injectivity, for each $(x_1,y_1) \in \pi(\Sigma \cap C)$, there exists a unique $(x_2,y_2)$ such that $(x_1,y_1,x_2,y_2) \in \Sigma \cap C$.  Moreover, by $(\Z/2)^2$-invariance of $\Sigma \cap C$, the points $(x_2,y_2,x_1,y_1)$, $(x_1,y_2,x_2,y_1)$, and $(x_2,y_1,x_1,y_2)$ also lie in $\Sigma \cap C$.  Thus for $(x_1,y_1) \in \pi(\Sigma \cap C)$ we define $f_1(x_1,y_1) = g_1(x_1,y_1,x_2,y_2)$.  Since $\Sigma \cap C$ is closed, we then extend $f_1$ arbitrarily to $M \times N$.

Now let us check that $\Phi_{f_1}$ satisfies the desired properties.  The value of $\Phi_{f_1}$ is irrelevant on points of $\Sigma^c \cap C$.  For $(x_1,y_1,x_2,y_2) \in \Sigma \cap C$, we compute
\begin{align*}
\Phi_{f_1}(x_1,y_1,x_2,y_2) & = f_1(x_1,y_1) + f_1(x_2,y_2) - f_1(x_1,y_2) - f_1(x_2,y_1) \\
& = g_1(x_1,y_1,x_2,y_2) + g_1(x_2,y_2,x_1,y_1) - g_1(x_1,y_2,x_2,y_1) - g_1(x_2,y_1,x_1,y_2) \\
& = 4g_1(x_1,y_1,x_2,y_2).
\end{align*}
Therefore $(\Phi_{f_1},g_2,\dots,g_d)$ is nonzero on $\Sigma \cap C$, since the original function $g$ was nonzero on $\Sigma \cap C$.

Now we address the assumption above.  For a generic function $g$, the set $\Sigma = (g_2,\dots,g_d)^{-1}(0)$ has codimension $d-1$, and hence dimension $2(p+q) - d + 1$, in $M \times N \times M \times N$.  Thus a generic function $\Sigma \to M \times N$ is an embedding when $2 \dim(\Sigma) + 1 \leq \dim(M \times N)$, that is,  when $\frac32 (p+q) + \frac32 \leq d$.  Thus in the dimension range of the conjecture, we may preempt the replacement of $g_1$ with a perturbation of $g$ on $C$, to ensure that the genericity assumption is satisfied.  Since $C$ is a closed set disjoint from $\Delta$, this perturbation of $g$ does not affect the condition $g^{-1}(0) = \Delta$.  Thus in the stated range of dimensions, the replacement of $g$ on $C$ proceeds inductively; at each stage $i$ we perturb the functions $f_1,\dots,f_{i-1}$ and $g_{i+1},\dots,g_d$, if necessary, so that the genericity assumption is satisfied, and then we replace $g_i$ with $\Phi_{f_i}$ as outlined above.

The issue with the argument above is that we have only replaced $g$ on $C$.  To replace $g$ on a neighborhood of $\Delta$, it is necessary to carefully study how the coupled embeddability condition manifests locally.  In the proof of Haefliger's embedding theorem, this replacement is made as follows.  An immersion is a local embedding, and $g$ is replaced near the diagonal by first using a result of Haefliger and Hirsch to homotope the $\Z/2$-equivariant map to an appropriate bundle monomorphism (possible in the stated dimension range), and then using the Smale--Hirsch theorem to homotope the monomorphism to an immersion.  After the replacement is made near the diagonal, the above replacement is made as outlined above.

Thus to complete the proof of this conjecture, it is necessary not only to understand how the coupled embedding condition manifests locally, but also to generalize both the Haefliger--Hirsch theorem and the Smale--Hirsch theorem.  We initiate this process with a brief study of the local condition.

\subsection{Local coupled embeddability}

Recall that a smooth map between smooth manifolds is called an immersion if the differential is a bundle monomorphism.  The immersion condition is first-order, and every immersion is a local embedding.  Intuitively, when considering the embedding condition on pairs of points $x_1, x_2 \in M$ which grow closer, the zero-order, two-point local embedding condition manifests as the first-order, single-point immersion condition.   The parallelogram condition is a condition at quadruples of points.  The goal of this section is to observe how locally this spawns a second-order, single-point differential condition.

\begin{definition} Given $(x,y) \in M \times N$, we say that $f \colon M \times N \to \R^d$ satisfies the \emph{local parallelogram condition at} $(x,y)$ if, for every neighborhood $U$ of $(x,y) \in M \times N$, the restriction $f \big|_U$ is not a coupled embedding.  Otherwise we say that $f$ is a \emph{local coupled embedding}.
\end{definition}

\begin{definition} Let $f \colon M \times N \to \R^d$ be a smooth map.   We say that $f$ is \emph{coupled nonsingular} if, at each $(x,y) \in M \times N$, in some local coordinates $(x_1,\dots,x_m)$ on $M$ and $(y_1,\dots,y_n)$ on $N$, $\frac{\partial^2 f}{\partial x_i \partial y_j} \neq 0$ for all $i, j$.  That is, the second derivative is nonzero for all pairs of the form $(u,0),(0,v)$, when considered as a symmetric bilinear map on $T_{(x,y)}(M \times N) \times T_{(x,y)}(M \times N)$.
\end{definition}

\begin{proposition}
\label{prop:nonlocalcoupled}
A coupled nonsingular map $f \colon M \times N \to \R^d$ is a local coupled embedding.
\end{proposition}

\begin{proof}
Suppose that $f$ fails to be a local coupled embedding at $(x,y)$.  Then in every neighborhood of $(x,y)$ there exist points $x_1, x_2 \in M$, $y_1, y_2 \in N$ which form an axis-aligned parallelogram.  We write the following:
\begin{align*}
f(x_1,y_1) - f(x,y) & = df_{(x,y)}(u_1,v_1) + \frac12 d^2f_{(x,y)}((u_1,v_1),(u_1,v_1)) + \cdots \\
f(x_2,y_2) - f(x,y) & = df_{(x,y)}(u_2,v_2) + \frac12 d^2f_{(x,y)}((u_2,v_2),(u_2,v_2)) + \cdots \\
f(x_1,y_2) - f(x,y) & = df_{(x,y)}(u_1,v_2) + \frac12 d^2f_{(x,y)}((u_1,v_2),(u_1,v_2)) + \cdots \\
f(x_2,y_1) - f(x,y) & = df_{(x,y)}(u_2,v_1) + \frac12 d^2f_{(x,y)}((u_2,v_1),(u_2,v_1)) + \cdots.
\end{align*}
Adding the first two equations and subtracting the latter two yields
\begin{align*}
f(x_1,y_1) + f(x_2,y_2) - f(x_1,y_2) - f(x_2,y_1) = d^2f_{(x,y)}((u_1-u_2,0),(0,v_1-v_2)) + \text{higher order terms}.
\end{align*}
The right side is nonzero by assumption, so there is no axis-aligned parallelogram.
\end{proof}

The fact that this condition is not coordinate-free lends some difficulty to a systematic study of this condition.  We expect that it would be difficult to develop an h-principle statement for the coupled nonsingularity condition, and hence to finish the outlined proof of Conjecture \ref{con:haef}.


\section{Coupled embeddings: nonexistence}
\label{sec:nonexistence}

In this section we show how exploiting the combinatorics of a triangulation of a space can yield insights into coupled embeddability.  To begin we review some definitions.

\subsection{Joins and deleted joins}
\label{sec:joins}
Recall that for topological spaces $X$ and $Y$ the \emph{join} is the space $X * Y$ obtained from $X \times Y \times [0,1]$ by taking the quotient with respect to the equivalence relation generated by $(x,y,0) \sim (x',y,0)$ for $x,x' \in X$ and $(x,y,1) \sim (x,y',1)$ for $y,y' \in Y$. If the simplicial complexes $\Sigma_X$ and $\Sigma_Y$ triangulate $X$ and~$Y$, respectively, then the join $X * Y$ is naturally triangulated by the \emph{join} of simplicial complexes~$\Sigma_X * \Sigma_Y$. As abstract simplicial complexes 
$$\Sigma_X * \Sigma_Y = \{\sigma \times \{1\} \cup \tau \times \{2\} \ : \ \sigma \in \Sigma_X, \ \tau \in \Sigma_Y\},$$
that is, the join is defined by the rule that the vertices of a face in $\Sigma_X$ and the vertices of a face in $\Sigma_Y$ together span a face in~$\Sigma_X * \Sigma_Y$. Here we assume that $\Sigma_X$ and $\Sigma_Y$ have disjoint vertex sets. 

One may think of the join $X * Y$ as abstract convex combinations of points in $X$ and in~$Y$. We will write $\lambda_1x+\lambda_2y$, with $\lambda_1, \lambda_2 \ge 0$ and $\lambda_1 + \lambda_2 = 1$, for the point $(x,y,\lambda_1)$ in~$X * Y$. Thus for $\lambda_1 = 0$ the point $x$ does not influence the point in~$X * Y$, whereas for $\lambda_1 = 1$ the choice of~$y$ does not matter. Note that in this notation $\lambda_1x_1+\lambda_2x_2$ and $\lambda_2x_2+\lambda_1x_1$ determine different points in $X * X$ if $x_1 \ne x_2$.

The \emph{deleted join} $X^{*2}_\Delta$ of a space $X$ is obtained from $X * X$ by deleting all points of the form $(x,x,t)$ with $x \in X$ and $t \in (0,1)$. The \emph{deleted join} $\Sigma^{*2}_\Delta$ of a simplicial complex $\Sigma$ is obtained from the join $\Sigma * \Sigma$ by deleting all those faces that have a vertex in $\Sigma$ in common, that is,
$$\Sigma^{*2}_\Delta = \{\sigma \times \{1\} \cup \tau \times \{2\} \ : \ \sigma, \tau \in \Sigma, \ \sigma \cap \tau = \emptyset\}.$$
In particular the deleted join $(\Delta_{n})^{*2}_\Delta$ of the $n$-simplex $\Delta_n$ is the boundary of an $(n+1)$-dimensional crosspolytope, and thus $(\Delta_{n})^{*2}_\Delta$ is homeomorphic to~$S^n$. Interchanging the join factors $\lambda_1x_1 + \lambda_2x_2 \mapsto \lambda_2x_2 + \lambda_1x_1$ is the antipodal action on the sphere~$(\Delta_{n})^{*2}_\Delta$.

\subsection{Discrete coupled embeddability}

Let $\Sigma$ be a simplicial complex.  A map $f \colon \Sigma \to \R^d$ is called an almost-embedding if for every pair of disjoint faces $\sigma_1, \sigma_2$ and every $x_1 \in \sigma_1$, $x_2 \in \sigma_2$, $f(x_1) \neq f(x_2)$.  We may discretize coupled embeddings similarly.

\begin{definition}  Let $\Sigma$ and $\Omega$ be simplicial complexes.  We say that $f \colon \Sigma \times \Omega \to \R^d$ has the \emph{discrete parallelogram condition} if there exist disjoint faces $\sigma_1, \sigma_2$ of $\Sigma$ and disjoint faces $\omega_1, \omega_2$ of $\Omega$, and $x_1 \in \sigma_1$, $x_2 \in \sigma_2$, $y_1 \in \omega_1$, $y_2 \in \omega_2$ which form an axis-aligned parallelogram.  Otherwise $f$ is a \emph{coupled almost-embedding}.  In the case of simplicial complexes, the notation $d(\Sigma,\Omega)$ refers to the smallest dimension admitting a coupled almost-embedding from $\Sigma \times \Omega$.
\end{definition}

To generate obstructions to coupled almost-embeddings, we require a statement on $(\Z/2)^2$-equivariant maps from $S^m \times S^n$ which takes into account different possible actions on the codomain.  Let $V_{--}, V_{+-}$, or $V_{-+}$ denote $\R$ as a $(\Z/2)^2$-module, where the subscript indicates the action of the two standard generators, that is, $V_{+-}$ indicates that the first generator acts trivially, while the second generator acts by negation.   The following lemma can be deduced easily from the work of Ramos~\cite{Ramos}.   A short proof using mapping degrees can be found in~\cite{FrickEtal}.

\begin{lemma}
\label{lem:productofspheres}
	Let $i, j, k, m,$ and $n$ be nonnegative integers with $m+n  = i+j+k$, $m \ge i$, and $n \ge j$. Suppose that $m-i$ and $n-j$ do not share a one in any digit of their binary expansions. Then any $(\Z/2)^2$-equivariant map $S^m \times S^n \to V_{-+}^i \times V_{+-}^j \times V_{--}^k$ has a zero.
\end{lemma}

\begin{lemma}
\label{lem:discemb}
If there exists a coupled almost-embedding $f \colon \Delta_m \times \Delta_n \to \R^d$, then there is a $(\Z/2)^2$-equivariant map $S^m \times S^n \to V_{-+} \times V_{+-} \times V_{--}^d$ which avoids zero. 
\end{lemma}

\begin{proof}
Given $f\colon \Delta_m \times \Delta_n \to \R^d$, define
\begin{align*}
\Phi_f & \colon (\Delta_m)^{*2}_\Delta \times (\Delta_n)^{*2}_\Delta \to V_{-+} \times V_{+-} \times V_{--}^d \\
     &\colon (\lambda_1x_1+\lambda_2x_2, \mu_1y_1+\mu_2y_2) \mapsto \\ & \ \ \ (\lambda_1-\lambda_2, \mu_1 - \mu_2, \lambda_1\mu_1 f(x_1,y_1) + \lambda_2\mu_2 f(x_2,y_2) - \lambda_1\mu_2 f(x_1,y_2) - \lambda_2\mu_1 f(x_2,y_1) ).
\end{align*}
Now recall from Section \ref{sec:joins} that the deleted join $(\Delta_{n})^{*2}_\Delta$ of the $n$-simplex $\Delta_n$ is the boundary of an $(n+1)$-dimensional crosspolytope, and thus $(\Delta_{n})^{*2}_\Delta$ is homeomorphic to~$S^n$.  Moreover, interchanging the join factors $\lambda_1x_1 + \lambda_2x_2 \mapsto \lambda_2x_2 + \lambda_1x_1$ is the antipodal action on the sphere~$(\Delta_{n})^{*2}_\Delta$.  Therefore $\Phi_f$ may be considered as a $(\Z/2)^2$-equivariant map on $S^m \times S^n$, which avoids zero due to the absence of axis-aligned parallelograms from vertex-disjoint faces. 
\end{proof}

The following corollary now results from combining Lemmas \ref{lem:productofspheres} and \ref{lem:discemb}. 

\begin{corollary}
Suppose that $m-1$ and $n-1$ do not share a one in any digit of their binary expansions.  Then there is no coupled almost-embedding $\Delta_m \times \Delta_n \to \R^{m+n-2}$.  That is, every map $f \colon \Delta_m \times \Delta_n \to \R^{m+n-2}$ admits an axis-aligned parallelogram.
\end{corollary}

\subsection{The coindex of embedding spaces}

Recall that the coindex of a free $\Z/2$-space $Z$ is the dimension of the largest-dimensional sphere which maps equivariantly into $Z$.  In \cite{FrickHarrison}, the authors studied the coindices of embedding spaces $\Emb(X,\R^d)$ (and of almost-embedding spaces $\AEmb(\Sigma,\R^d)$) with the compact-open topology and with the $\Z/2$-action given by $f \mapsto -f$.

To explain concretely, $\coind(\Emb(X,\R^d)) = q$ if the following two statements hold:
\begin{compactitem}
\item There exists a map $f \colon X \times S^q \to \R^d$ which is $\Z/2$-equivariant in the second factor, and for every $y \in S^q$, $f(-,y)$ is an embedding.
\item For every map $f \colon X \times S^{q+1} \to \R^d$ which is $\Z/2$-equivariant in the second factor, there exists $y \in S^{q+1}$ such that $f(-,y)$ is not an embedding.
\end{compactitem}

The study of the coindex of embedding spaces is also related to the studies of bilinear and biskew maps, and is related to coupled embeddings as follows:

\begin{proposition}
\label{prop:coupledcoindex}
If there exists a coupled embedding $X \times S^q \to \R^d$, then $\coind(\Emb(X,\R^d)) \geq q$.  
\end{proposition}

\begin{proof}
Suppose there exists a coupled embedding $f \colon X \times S^q \to \R^d$.  Then in particular there are no points $x_1,x_2,y$ satisfying $f(x_1,y) - f(x_1,-y) = f(x_2,y) - f(x_2,-y)$, and so the map
\[
X \times S^q \to \R^d \colon (x,y) \mapsto f(x,y) - f(x,-y)
\]
is $\Z/2$-equivariant in the second factor and an embedding of $X$ for every fixed $y$.  Therefore $\coind(\Emb(X,R^d)) \geq q$.
\end{proof}

The same statement holds for coupled almost-embeddings.  Thus the results of \cite{FrickHarrison} on coindices of embedding spaces can be used to generate coupled nonembeddability results when one of the factors is a sphere.  Instead of applying these results directly, we generalize the main results of \cite{FrickHarrison} to obtain results for coupled nonembeddability when neither factor is a sphere. 

We first require a definition and related lemma.  For a simplicial complex $\Sigma$ on ground set $[n]$, we denote by $\KG(\Sigma)$ the \emph{Kneser graph of its nonfaces}, that is, the graph whose vertices correspond to those subsets of $[n]$ that do not form a face of~$\Sigma$, and with edges between vertices corresponding to disjoint faces. For a graph $G$ we denote by $\chi(G)$ its \emph{chromatic number}, that is, the least number of colors $c$ needed to color its vertices such that the two endpoints of every edge receive distinct colors.  A proper coloring of the Kneser graph of minimal nonfaces induces a proper coloring of the Kneser graph of all nonfaces (by coloring some nonface $\sigma$ with the color of some minimal nonface $\tau \subset \sigma$).

The following was shown in \cite{FrickHarrison}, but the proof is short and we reproduce it here.

\begin{lemma}
\label{lem:color}
Let $\Sigma$ be a simplicial complex on ground set~$[n]$, and let $c = \chi(\KG(\Sigma))$. Then there is a map $\Psi\colon (\Delta_{n-1})^{*2}_\Delta \to \R^{c+1}$ with $\Psi(\lambda_1x_1 + \lambda_2x_2) = -\Psi(\lambda_2x_2 + \lambda_1x_1)$ such that $\Psi(\lambda_1x_1 + \lambda_2x_2) = 0$ implies $\lambda_1 = \frac12 = \lambda_2$ and $x_1, x_2 \in \Sigma$.
\end{lemma}

\begin{proof}
	Color the missing faces of $\Sigma$ by $\{1, 2, \dots, c\}$ in such a way that disjoint missing faces receive distinct colors. Let $\Sigma_j$ be the simplicial complex on ground set $[n]$ whose missing faces are the missing faces of $\Sigma$ colored~$j$. Then $\Sigma = \Sigma_1 \cap \dots \cap \Sigma_c$.
	
	Define the map $\Psi\colon (\Delta_{n-1})^{*2}_\Delta \to \R^{c+1}$ by $$\Psi(\lambda_1x_1+\lambda_2x_2) = (\lambda_1-\lambda_2, \lambda_1\mathrm{dist}(x_1,\Sigma_1)-\lambda_2\mathrm{dist}(x_2, \Sigma_1), \dots, \lambda_1\mathrm{dist}(x_1,\Sigma_c)-\lambda_2\mathrm{dist}(x_2, \Sigma_c)).$$
	
	Observe that $\Psi(\lambda_1x_1+\lambda_2x_2) = 0$ implies $\lambda_1 = \lambda_2 = \frac12$ and thus $\mathrm{dist}(x_1,\Sigma_j) = \mathrm{dist}(x_2,\Sigma_j)$ for all $j \in [c]$. Since by definition of $(\Delta_{n-1})^{*2}_\Delta$ the points $x_1$ and $x_2$ are in disjoint faces of~$\Delta_{n-1}$ and the missing faces of $\Sigma_j$ intersect pairwise, for every $j$ either $x_1 \in \Sigma_j$ or $x_2 \in \Sigma_j$. This implies $\mathrm{dist}(x_1,\Sigma_j) = \mathrm{dist}(x_2,\Sigma_j) =0$ and thus $x_1, x_2 \in \Sigma$.
\end{proof}

\begin{theorem}
\label{thm:alm}
For $i=1,2$, let $\Sigma_i$ be a simplicial complex on ground set $[n_i]$ and let $c_i = \chi(\KG(\Sigma_i))$.  Suppose that the nonnegative integers $n_1 - c_1 - 2$ and $n_2 - c_2 - 2$ do not share a one in any digit of their binary expansions.   Then there is no coupled almost-embedding $\Sigma_1 \times \Sigma_2 \to \R^{n_1+n_2-c_1-c_2-4}$.
\end{theorem}

\begin{proof}
Let $f \colon \Sigma_1 \times \Sigma_2 \to \R^d$, for $d = n_1+n_2-c_1-c_2-4$.  Extend $f$ continuously to $\Delta_{n_1-1} \times \Delta_{n_2-1}$.  Define
\begin{align*}
\Phi & \colon (\Delta_{n_1-1})_\Delta^{*2} \times (\Delta_{n_2-1})_\Delta^{*2} \to V_{-+}^{c_1+1} \times V_{+-}^{c_2+1} \oplus V_{--}^{d} \\
& \colon (\lambda_1 x_1 + \lambda_2 x_2, \mu_1 y_1 + \mu_2 y_2) \mapsto \\
& (\Psi_1(\lambda_1 x_1 + \lambda_2 x_2), \Psi_2(\mu_1 y_1 + \mu_2 y_2), \lambda_1\mu_1 f(x_1,y_1) + \lambda_2\mu_2 f(x_2,y_2) - \lambda_1\mu_2 f(x_1,y_2) - \lambda_2\mu_1 f(x_2,y_1) ),
\end{align*}
where $\Psi_1$ and $\Psi_2$ are the maps from Lemma \ref{lem:color}.  
Now by Lemma \ref{lem:productofspheres},  $\Phi$ must hit zero.  Therefore by Lemma \ref{lem:color}, $\lambda_1 = \lambda_2 = \mu_1 = \mu_2 = \frac12$ and $x_1, x_2 \in \Sigma_1$ and $y_1, y_2 \in \Sigma_2$.  The last $d$ components of $\Phi$ then imply that $f$ is not a coupled almost-embedding.
\end{proof}

As a corollary we obtain a rewording of Theorem 1.3 from \cite{FrickHarrison}, stated for coupled embeddings instead of for coindex.

\begin{corollary}
\label{cor:alm}
Let $\Sigma$ be a simplicial complex on ground set $[n]$, and let $p = n - \chi(\KG(\Sigma)) - 2$.   Suppose that the nonnegative integers $m$ and $p$ do not share a one in any digit of their binary expansions.  Then there is no coupled almost-embedding $\Sigma \times S^m \to \R^{m+p}$.  In particular, $d(\Sigma,S^m) > m+p$.
\end{corollary}

\begin{corollary} 
The following statements hold:
\begin{compactenum}\setlength{\itemsep}{.3cm}
\item $d(\R P^2,S^k) = 4\Big\lfloor \frac{k}{4} \Big\rfloor + 4$.
\item $d(\C P^2,S^{8q}) = 8q+7$, and $d(\C P^2,S^k) = 8\Big\lceil \frac{k}{8} \Big\rceil$ for $k \neq 8q$. 
\item $d(\Delta_{2k+2}^{(k)},S^1) = 2k+2$. 
\item $d([3]^{*(k+1)}, S^1) = 2k+2$
\item $d(\Sigma_{\R P^2}, \Delta_{4q+2}^{(2q)}) = 4q+4$.
\end{compactenum}
\end{corollary}

\begin{proof}
For each item the lower bounds are obtained by obstructing a coupled almost-embedding, and the upper bounds are obtained from nonsingular bilinear maps.  We make use of the fact that the real projective plane~$\R P^2$ can be triangulated in a unique way by a six-vertex triangulation, and the complex projective plane $\C P^2$ can be triangulated in a unique way by a nine-vertex triangulation.   These triangulations $\Sigma_{\R P^2}$ and $\Sigma_{\C P^2}$ each have the property that no two nonfaces are disjoint, and thus the Kneser graphs $\KG(\Sigma_{\R P^2})$ and $\KG(\Sigma_{\C P^2})$ have no edges. In particular, $\chi(\KG(\Sigma_{\R P^2})) = 1 = \chi(\KG(\Sigma_{\C P^2}))$.   For more details see Matou\v sek \cite[Example 5.8.5]{Matousek}.  \\
\begin{compactenum}\setlength{\itemsep}{.3cm}
\item The lower bounds come from Corollary \ref{cor:alm} applied to $m = 4q$ and to $\Sigma_{\R P^2}$,  so $n - \chi -2 = 3$.  The upper bounds come from the existence of bilinear maps $\R^4 \times \R^{4q+4} \to \R^{4q+4}$.

\item The lower bounds come from Corollary \ref{cor:alm} applied to $m = 8q$ and $m = 8q+1$ and to $\Sigma_{\C P^2}$,  so $n - \chi - 2  = 6$.    The upper bounds come from the existence of bilinear maps $\R^7 \times \R^{8q} \to \R^{8q}$ and $\R^7 \times \R^{8q+1} \to \R^{8q+7}$.

\item The lower bounds come from Corollary \ref{cor:alm} applied to $m=1$ and to $\Delta_{2k+2}^{(k)}$,  so $n - \chi - 2 = 2k$.  The upper bounds come from Corollary \ref{cor:dim}.

\item The lower bounds come from Corollary \ref{cor:alm} applied to $m=1$ and to $[3]^{*(k+1)}$,  so $n - \chi - 2 =2k$.  The upper bounds come from Corollary \ref{cor:dim}.

\item The lower bounds come from Theorem \ref{thm:alm} with $n_1 - c_1 - 2 = 3$ and $n_2 - c_2 - 2 = 4q$.  The upper bounds come from Corollary \ref{cor:dim}.
\end{compactenum}
This completes the proof.
\end{proof}

The proof of Theorem \ref{thm:mainobstruct} is similar.

\begin{proof}[Proof of Theorem \ref{thm:mainobstruct}]
The coupled embeddability of all pairs of simplicial complexes in dimension $2p+2q+1$ follows from Corollary \ref{cor:dim}.
The lower bounds come from Theorem \ref{thm:alm} with $n_1 - c_1 - 2 = 2p$ and $n_2 - c_2 - 2 = 2q$, which do not share any ones in their binary expansions by hypothesis.
\end{proof}

\bibliographystyle{plain}

\end{document}